\documentclass[a4paper,12pt]{amsart}
\usepackage{amsmath,amssymb,amsthm,amsfonts,xy,graphicx,tikz,ogonek}
\usetikzlibrary{arrows,calc}
\xyoption{all}
\numberwithin{equation}{section}

\newtheorem{definition}{Definition}[section]
\newtheorem{theorem}[definition]{Theorem}
\newtheorem{lemma}[definition]{Lemma}

\def\pr{\begin{proof}}
\def\sq{\end{proof}}

\newcommand{\cO}{\mathcal{O}}
\newcommand{\cP}{\mathcal{P}}

\newcommand{\Cpct}{{\mathcal{K}}}
\newcommand{\sw}[1]{{{}_{(#1)}}}

\newcommand{\qsph}{{ C(S^5_H)}}
\newcommand{\B}[1]{{\mathbb #1}}

\newcommand{\id}{\operatorname{id}}

\def\sw#1{{\sb{(#1)}}}

\newcommand{\tplz}{{\mathcal{T}}}
\newcommand{\T}{{\mathcal{T}}}

\newcommand{\alert}{}
\renewcommand{\phi}{\varphi}
\renewcommand{\epsilon}{\varepsilon}

\renewcommand{\[}{\begin{equation}}
\renewcommand{\]}{\end{equation}}

\begin{document}
\baselineskip=16pt
\author{Piotr~M.~Hajac}
\address{Instytut Matematyczny, Polska Akademia Nauk, ul.~\'Sniadeckich 8, Warszawa, 00-656 Poland} 
\email{pmh@impan.pl}
\author{Jan Rudnik}
\address{Instytut Matematyczny, Polska Akademia Nauk, ul.~\'Sniadeckich 8, Warszawa, 00-656 Poland}
\email{yarood@gmail.com}
\title[Bundles over the quantum complex projective plane]{\vspace*{-25mm}Noncommutative 
bundles over the multi-pullback quantum complex projective plane}
\maketitle
\begin{abstract}\vspace*{-7.5mm}
We equip  the multi-pullback $C^*$-algebra $C(S^5_H)$ of a noncommutative-deformation  of the 5-sphere
with a free $U(1)$-action, and show that its fixed-point subalgebra is isomorphic with the 
$C^*$-algebra of the multi-pullback quantum complex projective plane. Our main result 
is the stable non-triviality of the dual tautological line bundle associated to the action.
We prove it by combining Chern-Galois theory with the Milnor connecting homomorphism in $K$-theory.
 Using the Mayer-Vietoris six-term exact sequences and the functoriality
of the K\"unneth formula,  we also compute the $K$-groups
of~$C(S^5_H)$.
\end{abstract}\vspace*{-7.5mm}
\setcounter{tocdepth}{3}
{\footnotesize \tableofcontents}
\section*{Introduction}
This paper is a part of a bigger project  devoted to the $K$-theory of
multi-pullback noncommutative deformations of free actions on spheres defining complex and real 
projective spaces. The lowest-dimensional complex case is worked out in \cite{bhms05,hms06b} with the help
of index theory.  Herein we focus on the triple-pullback quantum complex projective 
plane~$\B CP^2_\T$~\cite{hkz12} and its quantum 5-sphere~$S^5_H$. 
Upgrading from pullback $C^*$-algebras of \cite{bhms05,hms06b} to triple-pullback $C^*$-algebras requires
a significant change of methods. In particular, we have to take care of the cocycle condition, as explained in
 Section~\ref{cocycle}, to compute the $K$-groups of  $C(S^5_H)$  and
$C(\B CP^2_\T)$ in Section~\ref{kgroups} and~\cite{r-j12} respectively.

The main theorem of the paper is:

\noindent\textbf{Theorem $\mathbf{2.4}$}
The section module $C(S^5_H)_u$ of the dual tautological line bundle 
over $\B CP^2_\T$ is \emph{not stably free} as a left 
$ C(\B CP^2_\T)$-module.

\noindent
The result is derived by comparing two idempotents: one coming from Chern-Galois theory applied
to the $U(1)$-action on $C(S^5_H)$,
 and the other one obtained by applying a formula \eqref{o2e} for the Milnor connecting homomorphism in a
$K$-theory exact sequence. It is the same strategy that was used to determine non-trivial generators of the $K_0$-group
of Heegaard quantum lens spaces~\cite{hrz13}.

To explain a wider background  and make the paper self-contained, we begin with a review
of basic building blocks that are subsequently assembled into new results. Concerning notation, we use the
unadorned tensor product $\otimes$ to denote the minimal (spatial) tensor product of $C^*$-algebras 
and
$\otimes_\mathrm{alg}$ to denote the algebraic tensor product.

\section{Preliminaries}

\subsection{From the Toeplitz algebra to quantum projective spaces}

\subsubsection{Toeplitz algebra}\label{toeplitz}

There are different ways to introduce the Toeplitz algebra~$\T$.
Herein we  define it as the universal $C^*$-algebra generated by 
one  isometry $s$, i.e.\ an element satisfying the relation 
\mbox{$s^*s=1$}. 
(Throughout the paper $s$ will always mean the generating isometry 
of~$\T$.) Likewise, $u$ will always mean the unitary element 
generating the $C^*$-algebra $ C(S^1)$ 
of all continuous complex-valued functions on the unit circle~$S^1:=\{x\in\mathbb{C}\;|\; |x|=1\}$. 
By mapping $s$ to $u$, 
we obtain the well-known short exact sequence of $C^*$-algebras~\cite{c-la1,c-la2}:
\begin{equation}
0\longrightarrow \Cpct\longrightarrow\tplz
\stackrel{\sigma}{\longrightarrow} C(S^1)
\longrightarrow 0.
\end{equation}

We consider the Toeplitz algebra as the $C^*$-algebra of continuous
functions on a \emph{quantum disc}. To justify this point of view,
we take the family of  universal $C^*$-algebras generated by
$x$ satisfying $x^*x-qxx^*=1-q$, $\|x\|=1$, $q\in [0,1]$ \cite{kl93}.
For $q\neq 1$, the norm condition is implied by the relation, and can
be omitted. For $q=1$, it yields precisely the $C^*$-algebra 
$ C(D)$ of all
continuous complex-valued functions on the unit disc $D:=\{x\in\mathbb{C}\;|\; |x|\leq 1\}$. Finally, for
$q=0$, we get the Toeplitz algebra. Thus we obtain the 
$\T$ as a $q$-deformation of~$ C(D)$.

Both the Toeplitz algebra $\T$ and $ C(S^1)$ are examples
 of  graph $C^*$-algebras~\cite{flr}. 
Graph $C^*$-algebras are generated by partial isometries. 
They come naturally equipped with a $U(1)$-action given by rephasing
these partial isometries by unitary complex numbers. This $U(1)$-action
is called the gauge action. A key feature of the symbol map $\sigma$
is that it is equivariant with respect to the gauge actions.

\subsubsection{Projective spaces}

Projective spaces of dimension $n\in\B N$ over a topological 
field $\B K$ are defined as follows:
\begin{gather}
\B KP^n:=\{(x_0,\dots,x_n)\in\B K^{n+1}\setminus(0,\dots,0)\}/\sim,\\
(x_0,\dots,x_n)\sim (y_0,\dots,y_n) \iff \exists\lambda\in\B K\setminus\{0\}:(x_0,\dots,x_n)=\lambda (y_0,\dots,y_n).\nonumber
\end{gather}
We denote the equivalence class of $(x_0,\dots,x_n)$ by 
$[x_0:\ldots:x_n]$.
There is the canonical affine open covering of the thus
 defined projective spaces:
\begin{equation}
\forall\;i\in\{0,\dots,n\}:\;
U_i:=\{[x_0:\ldots:x_n]\in\B KP^n\;|\ x_i\neq0\}\stackrel{\widetilde{\psi_{i}}}{\longrightarrow}\B K^n .
\end{equation}
The above homeomorphisms are given by
\[
\widetilde{\psi_{i}}([x_0:\ldots :x_n]):=\left(\frac{x_0}{x_i},\dots,\frac{x_{i-1}}{x_i},\frac{x_{i+1}}{x_i},\dots,\frac{x_n}{x_i}\right).
\]

Let us now focus our attention on
$\B K=\B C$. To express the covering subsets in  
$C^*$-algebraic terms, we choose closed rather than open coverings.
To this end, we define  the following closed refinement of the affine
covering:
\begin{equation}
\forall\;i\in\{0,\dots,n\}:\;
V_i:=\big\{[x_0:\ldots:x_n]\in\B CP^n\;|\ 
\alert{|x_i|=\max\{|x_0|,\ldots,|x_n|\}}\big\}\cong D^{n}.
\end{equation}
Here  
the homeomorphisms are given by appropriate restrictions of
$\widetilde{\psi_{i}}$'s denoted by~${\psi_{i}}$.
We  use the covering $\{V_i\}_{i}$
to present $\mathbb{C} P^n$ as a multi-pushout.
More precisely, we
pick indices $0\leq i<j\leq n$, 
 denote by $\psi_{ij}$ the restriction
of $\psi_{i}$  to $V_i\cap V_j$,
and take the following commutative diagram:
\begin{equation}\label{kpdiag}
\xymatrix{
 & & \mathbb{C}P^n & & \\
D^{n} \ar@{-->}[urr] &
V_i \ar[l]_-{\psi_i}\ar@{^{(}->}[ur] & &
V_j\ar[r]^-{\psi_j}\ar@{_{(}->}[ul] &
D^{n}\phantom{W}\ar@{-->}[ull]\\
D^{j-1}\times S^1\times D^{n-j} \ar@{^{(}->}[u] & &
V_i\cap V_j \ar@{_{(}->}[ul] \ar@{^{(}->}[ur] \ar[ll]_-{\psi_{ij}} \ar[rr]^-{\psi_{ji}} & &
D^{i}\times S^1\times D^{n-i-1}. \ar@<2ex>@{_{(}->}[u]
}
\end{equation}

\subsubsection{Quantum complex projective spaces}

Now we combine the foregoing presentation of projective spaces with
the idea that the Toeplitz algebra is the $C^*$-algebra of functions
on a quantum unit disc  to construct a new type of quantum projective spaces~\cite{hkz12}.  To define
them, first we excise from  diagram 
\eqref{kpdiag} its middle square, and dualise it to the multi-pullback
diagram of unital commutative 
$C^*$-algebras of functions on appropriate compact
Hausdorff spaces: 
\begin{equation}\label{dual}
\xymatrix{
 &  C(\mathbb{C}P^n)\ar@{-->}[dl]\ar@{-->}[dr] &  \\
 C(D)^{\otimes n}\ar@{->>}^{\pi^i_j}[d] & & 
 C(D)^{\otimes n}\ar@{->>}[d]\ar@{->>}[dll]_{\pi^j_i}\\
 C(D) ^{\otimes j-1}\!\otimes\!  C(S^1)\!\otimes\!
  C(D)^{\otimes n-j} & & 
 C(D)^{\otimes i}\!\otimes\!  C(S^1)\!\otimes\!
  C(D)^{\otimes n-i-1}\!\!\!\ar[ll]_{(\psi_{ji}\circ\psi^{-1}_{ij})^*}. 
}
\end{equation}
This yields a multi-pullback presentation of~$C(\B CP^n)$.  Then we leave $ C(S^1)$ unchanged and
 replace $ C(D)$ by~$\T$. 
It turns out that the formulae for 
$(\psi_{ji}\circ\psi^{-1}_{ij})^*$'s, $\pi^i_j$'s and $\pi^j_i$'s
 continue to make sense after these replacements, so that
quantum projective spaces can be defined as Pedersen's 
\emph{multi-pullback $C^*$-algebras} (see~\cite{p-gk99,cm00})
\[\label{pedersen}
B^\pi:=\left\{\left.(b_i)_i\in\prod_{i\in J}B_i\;\right|\;\pi^i_j(b_i)=\pi^j_i(b_j),\;\forall\, i,j\in J,\, i\neq j
\right\},
\]
where $\{\pi^i_j:B_i\rightarrow B_{ij}=B_{ji}\}_{i,j\in J,\,i\neq j}$ is 
the family of  $C^*$-homomorphisms defined through  commutative diagram 
(\ref{dual}) with $C(D)$ replaced by~$\T$.

\subsection{Mayer-Vietoris six-term exact sequence}

For a one-surjective  pullback diagram of $C^*$-algebras 
\[
\xymatrix{&A\ar[dl]\ar[dr]&\\B_0\ar@{>>}[dr]_{\pi_0}&&B_1,\ar[dl]^{\pi_1}\\&B_{01}&}
\]
there exists  the Mayer-Vietoris six-term exact sequence
 (e.g., see \cite[Theorem~21.2.2]{b-b98}
\cite[Section~1.3]{bhms05}, \cite{s-c84}):
\[\label{mv}
\xymatrix{K_0(A)\ar[r]&K_0(B_0\oplus B_1)\ar[r]&K_0(B_{01})\ar[d]^{\partial_{01}}\\
K_1(B_{01})\ar[u]^{\partial_{10}}&K_1(B_0\oplus B_1)\ar[l]&K_1(A).\ar[l]}
\]
In our applications of this exact sequence, we will need
explicit formulae for connecting homomorphisms
$\partial_{10}$ and~$\partial_{01}$.

\subsubsection{Odd-to-even connecting homomorphism}

Following the celebrated Milnor's construction of an odd-to-even
connecting homomorphism in algebraic $K$-theory~\cite{m-j71},
one can derive an explicit formula for this homomorphism 
\cite{r-a,dhhmw12}, and adapt it to unital $C^*$-algebras
(see \cite[Section~0.4]{hrz13} for an argument of Nigel Higson).
\begin{theorem}
Let $U\in \text{GL}_n(B_{01})$,  $(\id\otimes\pi_0)(c)=U^{-1}$ and 
$(\id\otimes\pi_0)(d)=U$.
Denote by $I_n$ the identity matrix of size~$n$, and put
\[\label{milnorid}
p_U:=\left(
\begin{array}{cc}
(c(2 - dc)d,1)&(c(2 - dc)(1 - dc),0)\\
((1 - dc)d,0)&((1 - dc)^2,0)
\end{array}\right)\in M_{2n}(A).
\]
Then $p_U$ is an idempotent and  the formula 
\[\label{o2e}
\partial_{10}([U]):=[p_U]-[I_n]
\]
defines an odd-to-even connecting homomorphism
 $\partial_{10}:K_1(B_{01})\to K_0(A)$ in the Mayer-Vietoris
six-term exact sequence~(\ref{mv}).
\end{theorem}

\subsubsection{Even-to-odd connecting homomorphism}\label{BHch}

Combining \cite[Theorem~1.18]{bm}
with \cite[Section~9.3.2]{b-b98}, we obtain:
\begin{theorem}
Let $p\in M_n(B_{01})$ be a projection, $(\id\otimes\pi_0)(Q_p)=p$, 
$Q_p^*=Q_p$, and $I_n$ be the identity matrix of size~$n$. 
Then  the formula  
\[\label{01}
\partial_{10}([p]):=[(e^{2\pi i Q_p},I_n)]
\]
defines an even-to-odd connecting homomorphism in
 the Mayer-Vietoris six-term exact 
se\-quen\-ce~\eqref{mv}.\vadjust{\goodbreak}
\end{theorem}

\subsubsection{Cocycle condition for multi-pullback $C^*$-algebras}\label{cocycle}

We  construct algebras of functions on quantum spaces
as multi-pullbacks of $C^*$-algebras. To make sure that this construction 
dually  corresponds to
the presentation of a quantum space as a ``union of closed subspaces'' (no self gluings of closed subspaces or their partial multi-pushouts;
see \cite{hz12} for an in-depth discussion of these issues), we assume
 the  cocycle condition. It allows us to apply the Mayer-Vietoris six-term exact sequence to multi-pullback $C^*$-algebras
by guaranteeing surjectivity of appropriate *-homomorphisms.

First we need some auxilliary definitions.
Let $(\pi^i_j:A_i\rightarrow A_{ij})_{i,j\in J,i\neq j}$ be a finite family of
surjective $C^*$-algebra homomorphisms. For all distinct $i,j,k\in J$, we define
$A^i_{jk}:=A_i/(\ker\pi^i_j+\ker\pi^i_k)$ and denote by \mbox{$[\cdot]^i_{jk}:A_i\rightarrow
A^i_{jk}$} the canonical surjections. Next, we introduce the family of maps
\begin{equation}
\pi^{ij}_k:A^i_{jk}\longrightarrow A_{ij}/\pi^i_j(\ker\pi^i_k),\qquad
[b_i]^i_{jk}\longmapsto\pi^i_j(b_i)+\pi^i_j(\ker\pi^i_k),
\end{equation}
for all distinct $i,j,k\in J$.
Note that they are isomorphisms when all $\pi^i_j$'s are surjective $C^*$-algebra homomorphisms, as assumed herein.

We say {\cite[in Proposition~9]{cm00}} that a finite family $(\pi^i_j:A_i\rightarrow A_{ij})_{i,j\in J,i\neq j}$ of $C^*$-algebra
surjections satisfies the {\em cocycle condition}
 if and only if, for all distinct $i,j,k\in J$,
\begin{enumerate}
\item $\pi^i_j(\ker\pi^i_k)=\pi^j_i(\ker\pi^j_k)$,
\item the isomorphisms
  $\phi^{ij}_k:=(\pi^{ij}_k)^{-1}\circ\pi^{ji}_k:A^j_{ik}\rightarrow A^i_{jk}$
  satisfy $\phi^{ik}_j=\phi^{ij}_k\circ\phi^{jk}_i$.
\end{enumerate}

One proves (\cite[Theorem~1]{hz12}) that a finite family $(\pi^i_j:A_i\rightarrow A_{ij})_{i,j\in J,i\neq j}$ of $C^*$-algebra
surjections satisfies the cocycle condition if and only if,
for all $K\subsetneq J$, $k\in J\setminus K$, and 
$(b_i)_{i\in K}\in \bigoplus_{i\in K}A_i$ such that $\pi^i_j(b_i)=\pi^j_i(b_j)$
for all distinct $i,j\in K$, there exists $b_k\in A_k$ such that
also $\pi^i_k(b_i)=\pi^k_i(b_k)$ for all $i\in K$. One can easily see that 
dually this corresponds to the statement ``a quantum space is a
pushout of parts, and all partial pushouts are embedded in this quantum space''.
This is  what we usually have in mind when constructing a space from parts.

\subsection{Actions of compact Hausdorff groups on unital $C^*$-algebras}


To use the language of strong connections \cite{h-pm96} and facilitate some 
computations, we need to transform actions of compact Hausdorff groups
on unital $C^*$-algebras 
to  coactions of their $C^*$-algebras on unital $C^*$-algebras. 
More precisely, 
let $A$ be a unital 
$C^*$-algebra and $G$ a compact Hausdorff group with a group homomorphism 
$\alpha\colon G\ni g\mapsto \alpha_g\in \mathrm{Aut}(A)$. Then
\[\label{aut}
\delta_\alpha\colon A\longrightarrow  C(G,A)\cong 
A\otimes  C(G),\quad \delta_\alpha(a)(g):=\alpha_g(a).
\]
We will use the thus related action and coaction interchangeably.

Furthermore, for any compact Hausdorff group $G$, we can define
the Hopf-algebraic structure on $ C(G)$ due to its
commutativity: 
\begin{itemize}
\item
the comultiplication~$\Delta\colon C(G)\to C(G)\otimes C(G)$, 
\item
the counit~$\varepsilon\colon C(G)\to\mathbb{C}$, 
\item
and the antipode $S\colon C(G)\to C(G)$
\end{itemize}
are respectively the pullbacks of the group
mulitplication,  the embedding of the neutral
element into $G$, and the inverting map
 $G\ni g\mapsto g^{-1}\in G$.
We can also use the Heynemann-Sweedler notation (with the summation
sign suppressed) for coactions  and
comultiplications:
\begin{itemize}
\item
$\delta(a)=:a_{(0)}\otimes a_{(1)},\quad\delta(a)(g)=(a_{(0)}\otimes a_{(1)})(g)=a_{(0)}a_{(1)}(g)$,
\item
$\Delta(h)=:h_{(1)}\otimes h_{(2)},\quad\Delta(h)(g_1,g_2)=(h_{(1)}\otimes h_{(2)})(g_1,g_2)=h_{(1)}(g_1)h_{(2)}(g_2)=h(g_1g_2)$.
\end{itemize}

In particular, for $G=U(1)$, 
 the antipode is determined by $S(u)=u^{-1}$, the counit by $\varepsilon(u)=1$, 
and finally the comultiplication by $\Delta(u)=u\otimes u$.  The coaction of $ C(U(1))$ on $\tplz$ coming from the aforementioned
(Section~\ref{toeplitz})
gauge action of $U(1)$ on $\tplz$ becomes
\begin{equation}
\label{gaugecoact}
\delta:\tplz\longrightarrow\tplz\otimes C(U(1)),
\quad \delta(s):=s\otimes u.
\end{equation}

\subsubsection{Freeness}

Following \cite{e-da00}, we say that an action of a compact Hausdorff group $G$ on a unital $C^*$-algebra $A$ is \emph{free}
if and only if the induced  coaction satisfies the following norm-density condition:
\[
\{(x\otimes 1)\delta(y)\;|\;x,y\in A\}^{\mathrm{cls}}=
A\otimes C(G).
\] 
Here ``cls" stands for ``closed linear span".

Next, let us denote by
 $\cO (G)$ the dense Hopf $*$-subalgebra  spanned by the matrix coefficients of finite-dimensional representations.  We define the
\emph{Peter-Weyl subalgebra} of $A$  as
\[
\cP_G(A):=\left\{\,a\in A\,\Big| \,\delta(a)\in A\underset{\mathrm{alg}}{\otimes}\cO (G)\,\right\}.
\]
One shows that it is an $\cO (G)$-comodule algebra which is a dense
$*$-subalgebra of~$A$. (See \cite{s-pm11} and references therein.) Moreover, the $C^*$-algebraic freeness condition on
a $G$-$C^*$-algebra $A$ 
is equivalent to the algebraic \emph{principality} condition on the $\cO (G)$-comodule algebra~$\cP_G(A)$~\cite{bdh}.
This allows us to use crucial algebraic tools without leaving
the ground of $C^*$-algebras.

\subsubsection{Strong connections and principal comodule algebras}

One can  prove (see \cite{bh} and references therein) that
a comodule algebra is principal if and only if it admits a strong 
connection. Therefore we will treat the existence of a strong connection 
as a condition defining the principality of a comodule algebra and avoid the original
definition of a principal comodule algebra. The latter is important when going beyond coactions that are
algebra homomorphisms --- then the existence of a strong connection is implied by
principality but we do not have the reverse implication~\cite{bh04}.

Let $G$ be a compact Hausdorff group acting on a unital $C^*$-algebra~$A$.
A \emph{strong connection $\ell$} on $A$ is a unital linear map 
$\ell :\mathcal{O}(G) \rightarrow \cP_G(A) \otimes_\mathrm{alg}\cP_G(A)$ satisfying:
\begin{enumerate}
\item 
$(\mathrm{id}\otimes \delta) \circ 
\ell = (\ell \otimes \mathrm{id}) \circ \Delta$,
$\big(((S\otimes\id)\circ\mathrm{flip}\circ\delta) \otimes \mathrm{id}\big) \circ 
\ell = (\mathrm{id} \otimes \ell) \circ
\Delta$;
\item 
$m \circ \ell=\varepsilon$, where 
$m\colon \cP_G(A)\otimes_\mathrm{alg}\cP_G(A)\to \cP_G(A)$ is the multiplication map.
\end{enumerate}
Here we abuse notation by using the same symbol for a restriction-corestriction of a map as for the map itself.

\subsubsection{Associated projective modules}

Let $\varrho\colon G\to GL(V)$ be a representation of
a compact Hausdorff group $G$ on a complex vector space~$V$, and  $\alpha\colon G\to\mathrm{Aut}(A)$ 
be an action on a unital $C^*$-algebra~$A$.
Then the \emph{associated module} $\cP_G(A)\Box^\varrho V$ is, by definition,
\[
\left\{x\in \cP_G(A)\underset{\mathrm{alg}}{\otimes} V\;\big|
\;\forall\; g\in G:\,(\alpha_g\otimes \id)(x)=\big(\id\otimes \varrho(g^{-1})\big)(x)\right\}.
\]
It is a left module over the fixed-point subalgebra 
$
A^\alpha:=\{a\in A\;|\;
\forall\,g\in G:\,\alpha_g(a)=a\}=:A^{U(1)}
$.

 If $V$ is finite dimensional and $\alpha$ is free, then $\cP_G(A)\Box^\varrho V$ is finitely generated projective~\cite{hm99}.
We think of it as the section module of an \emph{associated noncommutative vector bundle}.
 Furthermore, if $\dim V=1$ and $\gamma\colon G\to GL(\mathbb{C})$ is a representation, then we obtain:
\[
\cP_G(A)\overset{\gamma}{\Box}\mathbb{C}=\{a\in A\;|\;\delta(a)=a\otimes S(\gamma)\}=:A_{\gamma^{-1}}.
\]
\noindent
Modules $A_\gamma$ are called \emph{spectral subspaces}. We think of them as the section modules of 
	associated noncommutative \emph{line} bundles.

Now it is quite easy to apply Chern-Galois theory \cite[Theorem~3.1]{bh04}, and compute
 an idempotent $E_\gamma$ representing
 the associated module $A_\gamma$ using
 a strong connection~$\ell$:
\[\label{ie}
A_\gamma\cong (A^\alpha)^nE^\gamma,\quad E^\gamma_{ij}:=\gamma^R_i\gamma^L_j,
\quad \ell(\gamma)=:\sum_{k=1}^n \gamma^L_k\otimes\gamma^R_k
\in A_{\gamma^{-1}}\underset{\mathrm{alg}}{\otimes}A_\gamma,
\]
where $\{\gamma^L_k\}_k$ is a linearly independent set.

\subsubsection{Gauging coactions}\label{gauging}

Consider $A\otimes C(G)$ as a $C^*$-algebra
with the diagonal coaction 
\[
p\otimes h\longmapsto p\sw{0}\otimes h\sw{1}\otimes p\sw{1}h\sw{2}\,,
\]
 and denote by
$(A\otimes C(G))_R$ the same $C^*$-algebra but now equipped with the coaction on the
rightmost factor 
\[
p\otimes h\longmapsto p\otimes h\sw{1}\otimes h\sw{2}\,. 
\]
Then the following
map is a $G$-equivariant (i.e., intertwining the coactions) \emph{gauge} isomorphism of $C^*$-algebras:
\begin{equation}\label{kappa}
\widehat{\kappa}:(A\otimes C(G))\longrightarrow (A\otimes C(G))_R,\quad a\otimes h\longmapsto a\sw{0}\otimes a\sw{1}h.
\end{equation}
\noindent
Its inverse is explicitly given  by
\begin{equation}\label{kappa-1}
\widehat{\kappa}^{-1}:(A\otimes H)_R\longrightarrow (A\otimes H),\quad a\otimes h\longmapsto a\sw{0}\otimes S(a\sw{1})h.
\end{equation}

\section{Dual tautological line bundle}

\subsection{Quantum complex projective plane}
We consider the case $n=2$ of the multi-Toeplitz deformations
\cite[Section 2]{hkmz11}
of the complex projective spaces. The $C^*$-algebra of
our quantum projective plane is given as the  triple-pullback 
of the following diagram:
\begin{equation}\label{thefamily}
\xymatrix{
\mathcal{T}\otimes \mathcal{T}\ar[dr]^{\sigma_1}\ar@/_3pc/[ddrr]_{\sigma_2}&&\mathcal{T}\otimes 
\mathcal{T}\ar[dl]_{\Psi_{01}\circ\sigma_1}\ar[dr]^{\sigma_2}&&\mathcal{T}\otimes 
\mathcal{T}\ar[dl]_{\Psi_{12}\circ\sigma_2}\ar@/^3pc/[ddll]^{\Psi_{02}\circ\sigma_1}\\
&  C(S^1) \otimes \mathcal{T}&&\mathcal{T}\otimes  C(S^1) &\\
&&\mathcal{T}\otimes  C(S^1) &&}.
\end{equation}
Here 
$\sigma_1:=\sigma\otimes \id$, $\sigma_2:=\id\otimes\sigma$, and
\begin{align}\label{maps}
 C(S^1) \otimes \mathcal{T}\ni v\otimes t&\stackrel{\Psi_{01}}{\longrightarrow} S(t_{(1)}v)\otimes 
t\sw0 \in  C(S^1) \otimes \mathcal{T},\\
 C(S^1) \otimes \mathcal{T}\ni v\otimes t&\stackrel{\Psi_{02}}{\longrightarrow} t\sw0\otimes 
S(t_{(1)}v) \in \mathcal{T}\otimes  C(S^1) ,\nonumber\\
\mathcal{T}\otimes  C(S^1)  \ni t\otimes v&\stackrel{\Psi_{12}}{\longrightarrow}t\sw0\otimes 
S(t_{(1)}v)\in   \mathcal{T}\otimes  C(S^1) ,\nonumber
\end{align}
where $\mathcal{T}\ni t\mapsto t\sw0\otimes t_{(1)}\in 
\mathcal{T}\otimes  C(S^1) $ is the coaction of 
\eqref{gaugecoact}.

\subsection{Quantum complex projective plane $\B CP^2_\tplz$ as quotient  space $S^5_H/U(1)$}

Consider the following triple-pullback diagram in which every 
homomorphism is given by the symbol map on the appropriate factor
and identity otherwise:
$$
\xymatrix@R=40pt@C=0pt{
 C(S^1)\!\otimes\! \tplz\!\otimes\!\tplz\
\ar[dr]\ar@/_4pc/[ddrr]&&\tplz\!\otimes\!   C(S^1)\!\otimes\! \tplz
\ar[dl]\ar[dr]&&\tplz\!\otimes\! \tplz\!\otimes\!   C(S^1).\ar[dl]\ar@/^4pc/[ddll]\\
& C(S^1)\otimes  C(S^1)\otimes \tplz&&\tplz\otimes  C(S^1)\otimes  C(S^1) 
&\\
&&  C(S^1)\otimes\tplz\otimes  C(S^1)&&}
$$
\vspace*{-12.5mm}\[\label{s5}
\phantom{.}
\]
\smallskip

\begin{definition}\label{hs5}
The multi-pullback $C^*$-algebra of the  family of $C^*$-epimorphisms
in \eqref{s5} is called the $C^*$-algebra of the
\emph{Heegaard odd quantum sphere} $S^5_H$ and denoted~$ C(S^5_H)$.
\end{definition}

Using the coaction \eqref{gaugecoact} on $\T$ and the comultiplication on
$ C(S^1)=C(U(1))$, we define the diagonal coaction on each $C^*$-algebra
of the above diagram as in Section~\ref{gauging}. The diagram is
evidently equivariant with respect to this coaction because
the symbol map is equivariant. Therefore $ C(S^5_H)$ is a 
$U(1)$-$C^*$-algebra. We call this $U(1)$-action on 
$ C(S^5_H)$ \emph{diagonal}.

In order to compute the fixed-point subalgebra for the above diagonal $U(1)$-action,
we need to gauge it to an action on tensor products
that acts on the rightmost $ C(S^1)$-factor alone. Our goal is to show that the fixed-point subalgebra
is isomorphic with $ C(\B CP^2_\tplz)$.
To this end, we double
the three targets of all homomorphisms in \eqref{s5} to three pairs
of sibling targets, so that 
\[
\xymatrix{ C(S^1)\otimes\T\otimes\T\ar[rd]&&\T\otimes C(S^1)\otimes\T\ar[ld]\\
& C(S^1)\otimes C(S^1)\otimes\T&}
\]
becomes
\[
\xymatrix{ C(S^1)\otimes\T\otimes\T\ar[d]&\T\otimes C(S^1)\otimes\T\ar[d]\\
 C(S^1)\otimes C(S^1)\otimes\T& C(S^1)\otimes C(S^1)\otimes\T,\ar[l]_\id}
\]
and other subdiagrams are transformed in the same fashion.
Then we permute the factors in 
the tensor products in the top row  to make
 $ C(S^1)$ always the rightmost factor, and permute the target tensor products accordingly:
\[\label{permut}
\xymatrix{\T\otimes\T\otimes C(S^1)
\ar[d]_{\sigma\otimes\id\otimes\id}&\T\otimes\T\otimes C(S^1)
\ar[d]^{\sigma\otimes\id\otimes\id}\\
 C(S^1)\otimes\T\otimes C(S^1)& C(S^1)\otimes\T\otimes C(S^1).
\ar[l]_{T_{13}}
}\]
Here the horizontal arrow is just the flip of the outer factors.
Again, we apply analogous procedures to the other two subdiagrams. 
Due to the commutativity
of $ C(S^1)$, the thus obtained triple-pullback diagram
is equivariant for the diagonal coaction, and the $C^*$-algebra it defines is 
equivariantly isomorphic with the multi-pullback $C^*$-algebra defined by diagram~\eqref{s5}.

Now we are ready  to gauge the diagonal action as explained in Section~\ref{gauging}. 
Conjugating $T_{13}\circ(\sigma\otimes\id\otimes\id)$  by the gauge isomorphisms~\eqref{kappa}--\eqref{kappa-1}, 
using the commutativity and cocommutativity of $ C(S^1)=C(U(1))$,
along the lines of 
\cite[Section~5.2]{hkmz11}, we get:
\begin{align}
&\big(\widetilde{g}\circ T_{13}\circ(\sigma\otimes\id\otimes\id)\circ g^{-1}\big)
(r\otimes t\otimes w)
\nonumber\\ &=
\big(\widetilde{g}\circ T_{13}\circ(\sigma\otimes\id\otimes\id)\big)
(r\sw0\otimes t\sw0\otimes S(r\sw1 t\sw1)w)
\nonumber\\ &=
(\widetilde{g}\circ T_{13})(\sigma(r)\sw1\otimes t\sw0\otimes S(\sigma(r)\sw2 t\sw1)w)
\nonumber\\ &=
\widetilde{g}(S(\sigma(r)\sw2 t\sw1)w\otimes t\sw0\otimes\sigma(r)\sw1 )
\nonumber\\ &=
S(\sigma(r)\sw3 t\sw3)w\sw1\otimes t\sw0\otimes S(\sigma(r)\sw2 t\sw2)w\sw2 t\sw1 \sigma(r)\sw1
\nonumber\\ &=
S(\sigma(r) t\sw1)w\sw1\otimes t\sw0\otimes w\sw2.
\end{align}
Here $g$ and $\widetilde{g}$ are the gauge isomorphisms on $(\T\otimes\T)\otimes C(U(1))$
and  $( C(S^1)\otimes\T)\otimes C(U(1))$ respectively.

Much in the same way, we treat the remaining two subdiagrams of diagram~\eqref{s5}. 
Summarizing, for $0\leq i<j\leq 2$, the permuted and then gauged subdiagrams become:
\[\label{news5}
\xymatrix{
\T^{\otimes 2}\otimes C(S^1)\ar[d]_{\sigma_j}&\T^{\otimes 2}
\otimes C(S^1)\ar[d]_{\sigma_{i+1}}\\
\T^{\otimes j-1}\otimes C(S^1)\otimes\T^{\otimes 2-j}\otimes C(S^1)&
\T^{\otimes i}\otimes C(S^1)\otimes\T^{\otimes 1-i}\otimes C(S^1),\ar[l]_{\ \ \Psi^S_{ij}}
}
\]
where 
\begin{align}\label{psis}
\Psi^S_{01}(v\otimes t\otimes w)&:=S(vt_{(1)})w_{(1)}\otimes t_{(0)}\otimes w_{(2)},\\
\Psi^S_{02}(v\otimes t\otimes w)&:=t_{(0)}\otimes S(vt_{(1)})w_{(1)}\otimes  w_{(2)},\nonumber\\
\Psi^S_{12}(t\otimes v\otimes w)&:=t_{(0)}\otimes S(t_{(1)}v)w_{(1)}\otimes  w_{(2)}.\nonumber
\end{align}

The triple-pullback $C^*$-algebra of the family \eqref{news5} is denoted by 
$C(S^5_H)_R$. It is a $U(1)$-$C^*$-algebra
that is equivariantly isomorphic with $ C(S^5_H)$:
\begin{gather}\label{s5iso}
 C(S^5_H)\ni 
\big(v^0\otimes t^0\otimes r^0\;,\;t^1\otimes v^1\otimes r^1\;,\; t^2\otimes r^2\otimes v^2\big)
\longmapsto\\ 
\big(t^0\sw0\otimes r^0\sw0\otimes t^0\sw1 r^0\sw1 v^0 \;,\;{t^1}\sw0\otimes r^1\sw0\otimes t^1\sw1 
r^1\sw1 v^1\;,\; t^2\sw0\otimes r^2\sw0\otimes t^2\sw1 r^2\sw1 v^2\big)\in 
 C(S^5_H)_R .\nonumber
\end{gather}
This isomorphism yields an isomorhism of fixed-point subalgebras 
$ C(S^5_H)^{U(1)}\cong C(S^5_H)_R^{U(1)}$.
Since the $U(1)$-action in the triple-pullback diagram defining $ C(S^5_H)_R $
acts only on the rightmost factor, we conclude that $ C(S^5_H)_R^{U(1)}$
is the triple-pullback $C^*$-algebra obtained by
removing all rightmost factors in \eqref{news5} and taking $w=1$
in~\eqref{psis}.
Finally, since the isomorphisms in \eqref{psis} thus
become the isomorphisms
in~\eqref{maps}, so that \eqref{news5} becomes the defining 
triple-pullback diagram \eqref{thefamily}
 of $ C(\B CP^2_\tplz)$, we infer that 
$ C(S^5_H)^{U(1)}\cong C(\B CP^2_\tplz)$.

\subsection{Strong connection for the diagonal $U(1)$-action on
$C(S^5_H)$}

\begin{theorem}
The diagonal $U(1)$-action on $ C(S^5_H)$ is free.
\end{theorem}
\begin{proof}
We prove the claim by constructing a strong connection on the Peter-Weyl
comodule algebra $\mathcal{P}_{U(1)}( C(S^5_H))$ for the 
diagonal coaction 
$\delta\colon C(S^5_H)
\to C(S^5_H)\otimes C(U(1))$.
Let $u$ be the generating unitary of $ C(S^1)$ and $s$ be the 
generating isometry of $\tplz$.
Consider the following isometries in $ C(S^5_H)$:
\begin{align}
&a:=(u\otimes 1\otimes 1 , s\otimes 1\otimes 1, s\otimes 1\otimes 1),\\
&b:=(1\otimes s\otimes 1 , 1\otimes u\otimes 1, 1\otimes s\otimes 1),\nonumber\\
&c:=(1\otimes 1\otimes s , 1\otimes 1\otimes s, 1\otimes 1\otimes u).\nonumber
\end{align}
They all commute and satisfy the equation:
\[
(1-aa^*)(1-bb^*)(1-cc^*)=0.
\]

Now one can easily check that  a strong connection 
\[
\ell\colon\mathcal{O}(U(1))\longrightarrow
\mathcal{P}_{U(1)}( C(S^5_H))\underset{\mathrm{alg}}{\otimes}\mathcal{P}_{U(1)}( C(S^5_H))
\subseteq C(S^5_H)\otimes C(S^5_H)
\]
can be defined by the formulae:
\begin{align}
\ell(1)=1\otimes 1,&\qquad
\ell(u)=b^*\otimes b,\label{bb*}\\
\ell(u^*)=a\otimes a^* +b\otimes b^* +c\otimes c^* -a\otimes a^*&bb^*-a\otimes a^*cc^*-b\otimes b^*cc^*+a\otimes a^*bb^*cc^*.
\end{align}
Indeed, exactly as in
\cite[(4.6)]{hms06a}, we can show inductively that the formula
\[
\ell(x^n):=\ell(x)^{\langle 1 \rangle}\ell(x^{n-1})\ell(x)^{\langle 2 \rangle},\quad 
\ell(x)=:\ell(x)^{\langle 1 \rangle}\otimes\ell(x)^{\langle 2 \rangle}\text{ (summation suppressed),}
\]
has the desired properties for $x$ being any of the grouplikes
$u$ and $u^*$. 
\end{proof}

\subsection{Stable non-freeness}

\begin{definition}
Let $u$ be the generating unitary of $ C(U(1))$ and 
$ C(S^5_H)\stackrel{\delta}{\to} C(S^5_H)\otimes  C(U(1))$  the diagonal coaction. We call the associated module
$$
 C(S^5_H)_u:=\{x\in C(S^5_H)\;|\;\delta(x)=x\otimes u\}
$$
the section module of the \emph{dual tautological line bundle} over $\B CP^2_\T$.
\end{definition}

It follows from the existence of a strong connection 
on the Peter-Weyl comodule algebra $\mathcal P_{U(1)}( C(S^5_H))$
that $ C(S^5_H)_u$ is a finitely generated projective module over $ C(S^5_H)^{U(1)}\cong C(\B CP^2_\T)$~\cite{hm99}. Moreover, combining
\eqref{ie} with \eqref{bb*} proves that $ C(S^5_H)_u$
is isomorphic as a left {$ C(S^5_H)^{U(1)}$}-module with 
$ C(S^5_H)^{U(1)}bb^*$.
This allows us to prove our main result:
\begin{theorem}\label{Sc}
The section module $ C(S^5_H)_u$ of the dual tautological line bundle 
over $\B CP^2_\T$ is \emph{not stably free} as a left 
$ C(\B CP^2_\T)$-module.
\end{theorem}
\begin{proof}
The gauge isomorphism \eqref{s5iso} turns the projection $bb^*$ 
representing the finitely generated projective module
$ C(S^5_H)_u$  
to $(ss^*\otimes 1\;,\;1\otimes 1\;,\;1\otimes ss^*)\in C(\B CP^2_\T)$.
 Plugging it into the iterated pullback diagram
$$
\xymatrix{
&&& C(\mathbb C P^2_\T)\ar[dll]^\pi\ar[drr]&&\\
&P_1\ar[dr]\ar[drrr]\ar[dl]&&&&\mathcal{T}^{\otimes2}\ar[dl]\\
\mathcal{T}^{\otimes2}\ar[dr]_{\sigma_1}&&\mathcal{T}^{\otimes2}\ar[dl]^{{\Psi_{01}}\circ\sigma_1}&&P_{12}\ar[dl]\ar[dr]&\\
& C(S^1)\otimes\T &&\T\otimes C(S^1)\ar[dr]_{}&&\T\otimes C(S^1)\ar[dl]^{}\\
&&&& C(S^1)\otimes C(S^1)&}
$$
and projecting 
 via $\pi$ to $P_1$, we obtain $(ss^*\otimes 1\;,\;1\otimes 1)$.

Furthermore, consider the Mayer-Vietoris six-term exact sequence of the pullback diagram defining $P_1$, and take  unitary 
$u\otimes 1$ whose class generates $K_1( C(S^1)\otimes\T)$. We know 
from the proof of \cite[Theorem~2.1]{r-j12} that 
$K_0(P_1)=\B Z\oplus \B Z$ with one $\B Z$ generated by $[1]$ and the other $\B Z$ generated by $\partial_{10}([u\otimes 1])$. 
To compute the Milnor idempotent $p_{u\otimes 1}$ (see~\eqref{milnorid}), take a
 lifting of $u^{-1}\otimes 1$ to be $c:=s^*\otimes 1$,
 and a lifting of $u\otimes 1$ to be $d:=s\otimes 1$. Suppressing $\otimes 1$, we obtain 
\[
\partial_{10}([u\otimes 1])=\left[\left(
\begin{array}{cc}
(s^*(2 - ss^*)s,1)&(s^*(2 - ss^*)(1 - ss^*),0)\\
((1 - ss^*)s,0)&((1 - ss^*)^2,0)
\end{array}\right)\right]-[(1,1)]
=[(1-ss^*,0)].
\]
Hence $[1]-\partial_{10}([u\otimes 1])=[(ss^*\otimes 1,1\otimes 1)]=\pi_*[ C(S^5_H)_u]$.

Finally, if $ C(S^5_H)_u$ were stably free, then $\pi_*[ C(S^5_H)_u]=n[1]$ for some $n\in \B N$. This would contradict the just derived equality, so that $ C(S^5_H)_u$ is not stably free. 
\end{proof}

\section{$K$-groups of the quantum sphere $S^5_H$}\label{kgroups}

We end this paper by showing that  the $K$-groups of 
$S^5_H$ agree with  their 
classical counterparts. Its $C^*$-algebra is the triple-pullback  $C^*$-algebra
(see  \ref{hs5}),
so that we can apply \cite[Corollary~1.5]{r-j12} to determine its $K$-theory.

\subsection{Cocycle condition}

The first step in applying \cite[Corollary 1.5]{r-j12} is verifying the cocycle condition (see Section~\ref{cocycle}). 
\begin{lemma}\label{s5co}
The family~\eqref{s5} defining the triple-pullback $C^*$-algebra $ C(S_H^5)$
satisfies the cocycle condition.
\end{lemma}
\begin{proof}
It is straightforward to check the first part of the cocycle condition. We do it only in one case as all
other cases are completely analogous. For $i=2 $, $j=1$ and $k=0$, we obtain:
\begin{align}
\pi^2_1(\ker\pi^2_0)&=\pi^1_2(\ker\pi^1_0)\Leftrightarrow\nonumber\\
\sigma_2(\ker\sigma_1)&=\sigma_3(\ker\sigma_1)\Leftrightarrow\nonumber\\
\sigma_2(\mathcal K\otimes \tplz\otimes  C(S^1))&=\sigma_3(\mathcal K\otimes \mathcal 
C(S^1)\otimes \tplz))\Leftrightarrow\nonumber\\
\mathcal K\otimes  C(S^1)\otimes C(S^1)&=\mathcal K\otimes \mathcal 
C(S^1)\otimes C(S^1).
\end{align}

For the second part we use the following notation
\begin{equation}\label{bracket}
[\cdot]^i_{jk}:B_i\rightarrow B_i\slash (\ker\pi^i_j+\ker\pi^i_k),\quad
[\cdot]^{ij}_k:B_{ij}\rightarrow B_{ij}\slash \pi^i_j(\ker\pi^i_k).
\end{equation}
 Again all cases are done in a similar way, so that we only  check  the case 
$i=0$, $j=1$, $k=2$, i.e.\ we show that
$\phi^{02}_1=\phi^{01}_2\circ\phi^{12}_0$.
For any  $r\otimes  t\otimes v\in\T\otimes\T\otimes C(S^1)$, the left hand side is:
\begin{align}
\phi^{02}_1\big([r\otimes  t\otimes v]_{01}^2\big)
&=
\big((\pi_1^{02})^{-1}\circ\pi_1^{20}\big)\big([r\otimes  t\otimes v]_{01}^2\big)
\nonumber\\ &=
\Big[\big((\pi_2^{0})^{-1}\circ\pi_0^{2}\big)(r\otimes  t\otimes v)\Big]^0_{21}
\nonumber\\ &=
\big[(\sigma_3^{-1}\circ\sigma_1)(r\otimes  t\otimes v)\big]^0_{21}
\nonumber\\ &=
\big[\sigma(r)\otimes  t\otimes \omega(v)\big]^0_{21},
\end{align}
where $\omega$ is a linear splitting of~$\sigma$. On the other hand, we obtain:
\begin{align}
(\phi^{01}_2\circ\phi^{12}_0)\big([r\otimes  t\otimes v]_{01}^2\big)
 &=
\phi^{01}_2\Big(\Big[\big((\pi_2^{1})^{-1}\circ\pi_1^{2}\big)(r\otimes  t\otimes v)\Big]_{20}^1\Big)
\nonumber\\ &=
\phi^{01}_2\big(\big[r\otimes \sigma(t)\otimes \omega(v)\big]_{02}^1\big)
\nonumber\\ &=
\big((\pi^{01}_2)^{-1}\circ\pi^{10}_2\big)\big(\big[r\otimes \sigma(t)\otimes \omega(v)\big]_{02}^1\big)
\nonumber\\ &=
\Big[\big((\pi^0_1)^{-1}\circ\pi^1_0\big)\big(r\otimes \sigma(t)\otimes \omega(v)\big)\Big]^0_{12}
\nonumber\\ &=
\big[\sigma(r)\otimes \omega(\sigma(t))\otimes \omega(v)\big]^0_{12}
\nonumber\\ &=
[\sigma(r)\otimes t\otimes \omega(v)]^0_{12}.
\end{align}
Hence the left and the right hand side agree because $[]^i_{jk}=[]^i_{kj}$ for any set of distinct indices.
\end{proof}

\subsection{$K$-groups}

We are now ready for:
\begin{theorem}
 The $K$-groups of the Heegaard quantum $5$-sphere are:
$$
K_0(\qsph)=\B Z=K_1(\qsph).
$$
\end{theorem}
\pr
Lemma~\ref{s5co} allows us to apply \cite[Corollary 1.5]{r-j12} to the family of surjections in  diagram~\eqref{s5}.
The first six-term exact sequence is
\begin{equation}\label{i}
\xymatrix{
K_0(P_1)\ar[r]&K_0(\tplz^{\otimes 2}\otimes C(S^1)) \oplus K_0(\tplz\otimes C(S^1)\otimes\tplz)\ar@{.>}[r]
&K_0(\tplz\otimes C(S^1)^{\otimes 2})\ar[d]^{\partial_{01}}\\
K_1(\tplz\otimes C(S^1)^{\otimes 2})\ar[u]
&K_1(\tplz^{\otimes 2}\otimes C(S^1)) \oplus K_1(\tplz\otimes C(S^1)\otimes\tplz)\ar@{.>}[l]&K_1(P_1).\ar[l]}
\end{equation}
Here the dotted arrows are $(\id\otimes\sigma\otimes\id)_*-(\id\otimes\id\otimes\sigma)_* $.
With the help of  the K\"unneth formula the exact sequence becomes:
\begin{equation}
\xymatrix@C+50pt{
K_0(P_1)\ar[r]&\B Z \oplus\B Z \ar@{.>}[r]^{(m,n)\mapsto(m-n,0)}&\B Z \oplus\B Z \ar[d]\\
\B Z \oplus\B Z \ar[u]&\B Z \oplus\B Z \ar@{.>}[l]_{(m,n)\mapsto(m,-n)}&K_1( P_1).\ar[l]}
\end{equation}
Hence
$
K_0(P_1)=\B Z=K_1(P_1)
$.

The second diagram of \cite[Corollary 1.5]{r-j12} is
\begin{equation}\label{ii}
\xymatrix{
K_0(P_2)\ar[r]&K_0( C(S^1)\otimes\tplz\otimes  C(S^1)) \oplus K_0( C(S^1)^{\otimes 2}\otimes\tplz)\ar@{.>}[r]
&K_0( C(S^1)^{\otimes 3})\ar[d]^{\partial_{01}}\\
K_1( C(S^1)^{\otimes 3})\ar[u]^{\partial_{10}}
&K_1( C(S^1)\otimes\tplz\otimes  C(S^1)) \oplus K_1( C(S^1)^{\otimes 2}\otimes\tplz)\ar@{.>}[l]&K_1(P_2).\ar[l]}
\end{equation}
In order to unravel this diagram, we need to take a closer look into the K\"unneth formula.
We consider $[u]\otimes[u]\in K_1( C(S^1))\otimes K_1( C(S^1))$ and
denote its image under
the K\"unneth isomorphism
$K_1( C(S^1))\otimes K_1( C(S^1))\to 
K_0( C(S^1)\otimes C(S^1))$ by~$\beta$.
Using the natural leg numbering convention, we extend this notation to triple tensor products
with $[1]\in K_0(\tplz)$ or with $[1]\in K_0( C(S^1))$ as an appropriate factor. 
Next, we denote by $u_i$ the $K_1$-class of a triple tensor with $u$ as the $i$-th factor
and $1\in\tplz$ or $1\in C(S^1)$
 as any remaining factor. Hence $K_1( C(S^1)^{\otimes 3})$ is $\B Z^4$ generated by
$u_1$, $u_2$, $u_3$ and the fourth generator denoted by $u_{123}$. Furthermore, 
 the above exact sequence becomes
\begin{equation*}
\xymatrix@C+10pt{
K_0(P_2)\ar[r]&
\B Z [1]\oplus\B Z \beta_{13}\oplus\B Z [1]\oplus\B Z \beta_{12}
\ar@{.>}[r]&
\B Z [1]\oplus\B Z \beta_{12}\oplus\B Z \beta_{13}\oplus\B Z \beta_{23}
\ar[d]\\
\B Z u_1\oplus\B Z u_2\oplus\B Z u_3\oplus\B Z u_{123}
\ar[u]&
\B Z u_1\oplus\B Z u_3\oplus\B Z u_1\oplus\B Z u_2
\ar@{.>}[l]&K_1( P_2).\ar[l]}
\end{equation*}
Now, by the functoriality of the  K\"unneth isomorphism
\cite[p.~232]{b-b98} and with the help of the diagram~\eqref{iterd}, 
it is straightforward to verify that the upper and the lower
dotted maps are respectively given by
\[
{(a,b,c,d)\longmapsto(a-c,-d,b,0)}\quad\text{and}\quad{(a,b,c,d)\longmapsto(a-c,-d,b,0)}.
\]
Hence, by a straightforward homological computation, we infer that
\[
K_0(P_2)=\B Z [1]\oplus \B Z\,\partial_{10}(u_{123})\quad\text{and}\quad 
K_1(P_2)=\B Z[(\mathrm{u}_1,\mathrm{u}_1)]\oplus \B Z\,\partial_{01}(\beta_{23}),
\]
where $\mathrm{u}_1:=u\otimes 1\otimes 1$.

Finally, the last diagram of \cite[Corollary 1.5]{r-j12} is
\begin{equation}\label{iii}
\xymatrix{
K_0(\qsph)\ar[r]&K_0(P_1) \oplus K_0( C(S^1)\otimes\tplz^{\otimes 2})\ar@{.>}[r]&K_0(P_2)\ar[d]\\
K_1(P_2)\ar[u]&K_1(P_1) \oplus K_1( C(S^1)\otimes\tplz^{\otimes 2})\ar@{.>}[l]&K_1(\qsph).\ar[l]}
\end{equation}
Plugging in generators into this diagram, we obtain
\begin{equation}
\xymatrix@C+20pt{
K_0(\qsph)\ar[r]&\B Z[1] \oplus\B Z[1] \ar@{.>}[r]&
\B Z[1] \oplus\B Z\,\partial_{10}(u_{123})\ar[d]\\
\B Z[(\mathrm{u}_1,\mathrm{u}_1)]\oplus \B Z\,\partial_{01}(\beta_{23})
\ar[u]&\B Z\,\partial_{01}(\beta_{23}) \oplus\B Zu_1\ar@{.>}[l]&K_1(\qsph).\ar[l]}
\end{equation}
Here the upper dotted arrow is evidently given by the formula
${(a,b)\mapsto(a-b,0)}$. It is a bit more complicated to determine the lower dotted arrow.
To this end, we denote by $\mathrm{b}\in M_2\big( C(S^1)\otimes  C(S^1)\big)$
the pullback of the Bott projection on $S^2$, so that $[\mathrm{b}]=\beta$. Next, by 
$\bar{\mathrm{b}}\in M_2\big(\tplz\otimes  C(S^1)\big)$ we denote a self-adjoint lifting 
of~$\mathrm{b}$ along $\id_{M_2(\B C)}\otimes(\sigma\otimes\id)$.
Then we substitute $\bar{\mathrm{b}}$ to the formula \eqref{01} to compute both   
$\partial_{01}(\beta_{23})\in K_1(P_1)$ and $\partial_{01}(\beta_{23})\in K_1(P_2)$ at the
same time. The resulting formulas will only differ in the leftmost tensor factor: for $P_1$ it will
be $1\in\tplz$ and for $P_2$ it will be $1\in C(S^1)$. 
Therefore 
\[
 (\sigma_1,\sigma_1)_*\colon K_1(P_1)\ni\partial_{01}(\beta_{23})\stackrel{}{\longmapsto} \partial_{01}(\beta_{23})\in K_1(P_2).
\]
Combining this observation with
the diagram 
\begin{equation}\label{iterd}
\xymatrix@-10pt@C=-10pt{
&&&{\widetilde P}\ar[dll]\ar[drr]&&\\
&\widetilde P_1\ar[dr]\ar[drrr]^{\widetilde\gamma}\ar[dl]&&&&B^\pi/I_2\ar[dl]_{\widetilde\delta}\\
B^\pi/I_0\ar[dr]&&B^\pi/I_1\ar[dl]&&\widetilde P_2\ar[dl]\ar[dr]&\\
&B^\pi/(I_0+I_1)&&B^\pi/(I_0+I_2)\ar[dr]&&B^\pi/(I_1+I_2),\ar[dl]\\
&&&&B^\pi/(I_0+I_1+I_2)&}
\end{equation} 
one easily checks that the desired lower dotted map is given by the
formula ${(a,b)\mapsto(-b,a)}$. Consequently,
$
K_0(\qsph)=\B Z= K_0(\qsph)
$
as claimed.
\sq

\section*{Acknowledgements}
\noindent
The authors are extremely grateful to Paul Baum and Nigel Higson for pivotal $K$-consultations. The writing up of this paper was
partially supported by 
NCN grant 
2012/06/M/ST1/00169.

\end{document}